\documentclass[12pt, reqno]{amsart}
\usepackage{amsmath, amsthm, amscd, amsfonts, amssymb, graphicx, color}
\usepackage[bookmarksnumbered, colorlinks, plainpages]{hyperref}
\usepackage{tikz}
\usetikzlibrary{shapes, arrows}
\usepackage{tkz-graph}
\usepackage[all]{xy}
\graphicspath{{fig/}}
\usepackage{subfig}
\usepackage{verbatim}

\makeatletter \oddsidemargin.9375in \evensidemargin \oddsidemargin
\def\authorsaddresses#1{\dedicatory{#1}}
\newtheorem{theorem}{Theorem}[section]

\newtheorem{corollary}[theorem]{Corollary}
\theoremstyle{definition}

\theoremstyle{remark}

\numberwithin{equation}{section}

\begin{document}
\setcounter{page}{1}



\title[Contraction of graph and spanning trees]{Contraction of graphs and spanning k-end trees}

\author[GHASEMIAN ZOERAM]{HAMED GHASEMIAN ZOERAM}

\authorsaddresses{ Department of Pure Mathematics, Ferdowsi University
of Mashhad,  Iran\\ 
hamed90ghasemian@gmail.com\\
}


\begin{abstract}
A tree with at most $k$ leaves is called $k$-ended tree, and a tree with exactly $k$ leaves is called $k$-end tree, where a leaf is a vertex of degree one. 
Contraction of a graph $G$ along the edge $e$ means deleting the edge $e$ and identifying its end vertices and deleting all edges between every two vertex except one edge to gain again a simple graph. Contraction of edge $e$ on graph $G$ is denoted by $G/ e$.
In this paper we prove some theorems related to a graph and its contraction. For example we prove the following theorem. If $G$ is a connected graph that has a spanning $k$-end tree and $|V(G)| >K+1$ then there exist an edge $e$ such $G / e$ has a spanning $k$-end tree.
\end{abstract}

\maketitle


\section{Introduction}

In this paper all graphs are simple. Vertex set and edge set of graph $G$ are denoted by $V(G)$ and $E(G)$ successively, degree of vertex $v$ in graph $G$ is denoted by $\deg_{G}(v)$. If $v$ is a vertex of graph $G$, $N_G(v)$ is set of all vertices adjacent to $v$ in $G$. If $T$ is a tree then the unique path between very two vertices $v$ and $u$ is denoted by $uTv$ and $u^{-}$ is the vertex adjacent to $u$ in this path, and also $v^{-}$. We also denote the edge $e$ with $uv$ or $vu$ where $u$ and $v$ are end vertices of $e$. A Hamiltonian path in graph $G$ is a path that contains all vertices of graph. Subdividing the edge $e$ with end vertices $u$ and $v$ in graph $G$ is an operation that produces a new graph whose vertex set is $V(G)\cup\{w\}$ and edge set is $(E(G)-\{e\})\cup\{e^{'},e^{''}\}$ where $w$ is a new vertex, $e^{'}=uw$ and $e^{''}=wv$(figure 1.1). A graph that is obtained of finite sequence of subdivisions of edges of graph $G$ is called a subdivision of $G$.
\begin{figure}[h!] 
\centering
\begin{tikzpicture}[scale=1] 
\GraphInit[vstyle=Classic]
\tikzset{VertexStyle/.style = {shape = circle,fill = black, minimum size = 1pt,inner sep=1.5pt}}
\Vertex[Lpos=90,L=$u$,x=-6, y=0]{x}
\Vertex[Lpos=90,L=$v$,x=-3, y=0]{y}
\Vertex[Lpos=90,L=$u$,x=-0.5, y=0]{z}
\Vertex[Lpos=90,L=$w$,x=1, y=0]{w}
\Vertex[Lpos=90,L=$v$,x=2.5, y=0]{p}
\Edges(x,y)
\Edges(z,w,p)
\end{tikzpicture}
\caption*{Figure 1.1: right side graph is a subdivision of left side}\label{fi2.1}
\end{figure}
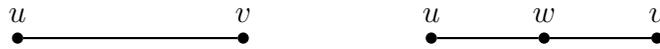
\\Now we presents some famous theorems about $k$-ended trees. First one is Ore's theorem.
\begin{theorem}
Suppose $G$ be a graph with $|V(G)|\geq3$, if for every two non-adjacent vertices $v$ and $w$ of $G$ we have $\deg {v} +\deg {w} \geq {|V(G)|}-1$ then $G$ has a spanning 2-ended tree or a Hamiltonian path.
\end{theorem}
A graph is called $K_{1,4}$-free if doesn't contain $K_{1,4}$ as a induced sub graph where $K_{1,4}$ is complete bipartite graph illustrate in figure 1.2. If $k$ is a positive integer then $\sigma_{k}(G)$ =min\{$\deg(U)$ : U is an independent set of $G$ with $|U|=k$\}, where $\deg(U)=\sum_{u\in U} \deg (u)$\\
\begin{figure}[h!] 
\centering
\begin{tikzpicture}[scale=1] 
\GraphInit[vstyle=Classic]
\tikzset{VertexStyle/.style = {shape = circle,fill = black, minimum size = 1pt,inner sep=1.5pt}}
\Vertex[Lpos=90,L=$$,x=2, y=3]{x}
\Vertex[Lpos=90,L=$$,x=2, y=1.8]{y}
\Vertex[Lpos=90,L=$$,x=2, y=0.6]{z}
\Vertex[Lpos=90,L=$$,x=2, y=-0.6]{w}
\Vertex[Lpos=90,L=$$,x=-1, y=1.2]{p}
\Edges(p,x)
\Edges(p,y)
\Edges(p,z)
\Edges(p,w)
\end{tikzpicture}
\caption*{Figure 1.2}\label{fi2.1}
\end{figure}
\begin{theorem}
Every connected $K_{1,4}$-free graph $G$ with $\sigma_4\ge|G|-1$ contains a spanning tree with at most $k$ leaves.
\end{theorem}

\section{Main results}

If $e$ is an edge of graph $G$ with end vertices $u$ and $v$ then we define $N_e(u)=\{x\in V(G); xu\in E(G), x\neq v\}$ and so $N_e(v)=\{x\in V(G); xv\in E(G), x\neq u\}$.

\begin{theorem}
Suppose $G$ is a connected graph and has a spanning $k$-end tree($k$$\in$$\mathbb{N}$,$k$$\geq$2), and $|V(G)|>k+1$. Then there exist an edge $e$ such $G/e$ has a spanning $k$-end tree.
\end{theorem}
At theorem $2.1$ we cannot choose an arbitrary edge and make contraction on graph with that edge to get the theorem $2.1$ result. For example, if we consider the graph in figure 2.1 and $e$=$py$ then it has a spanning $3$-end tree but $G/e$ has just a spanning $4$-end tree.\\

\begin{figure}[h!] 
\centering
\begin{tikzpicture}[scale=1] 
\GraphInit[vstyle=Classic]
\tikzset{VertexStyle/.style = {shape = circle,fill = white, minimum size = 1pt,inner sep=1.5pt}}
\Vertex[Lpos=180,L=$G:$,x=-3.5, y=0]{g}
\tikzset{VertexStyle/.style = {shape = circle,fill = black, minimum size =1pt,inner sep=2pt}}
\Vertex[Lpos=90,L=$x$,x=-2, y=0]{x}
\Vertex[Lpos=270,L=$y$,x=-0.4, y=0]{y}
\Vertex[Lpos=270,L=$z$,x=1, y=0]{z}
\Vertex[Lpos=270,L=$w$,x=2.5, y=0]{w}
\Vertex[Lpos=0,L=$p$,x=2.5, y=1.5]{p}
\Vertex[Lpos=0,L=$u$,x=3.5, y=3]{u}
\Vertex[Lpos=0,L=$v$,x=1.5, y=3]{v}
\Edges(x,y,z,w,p,u)
\Edges(y,p,v)
\end{tikzpicture}
\caption*{Figure 2.1}\label{fi2.1}
\end{figure}
\begin{proof}
Suppose $T$ is a spanning $k$-end tree of $G$, put:\\
$A=\{x\in {V(G)} ; \deg_T (x) \neq 1\}$ if $|A|\le 1$ then because $|V(G)|>k+1$ so the number of vertices with degree one in $T$ is greater than $k$, and this is contradiction. So $|A|$$>$$1$, now choose two different vertices $u,v\in A$, in the unique path $uTv$ if consider $uu^-$ then $T/uu^-$ is a spanning tree of $G/uu^-$ with $k$ leaves.
\end{proof}


\begin{theorem}
Consider a connected graph $G$ that has a spanning $k$-end tree with $k$$\ge$$3$ then there exist a sequence $e_1,e_2,\ldots,e_m$ of edges in $G$($m\in\mathbb{N}$) such, if put $G_1=G/e_1$ and $G_i=G_{i-1}/e_i$ $(i=2,...,m)$ then $G_m$ has a spanning $k-1$-end tree.
\end{theorem}

\begin{proof}
Suppose $T$ is a spanning $k$-end tree of $G$. We choose a vertex $v$ of degree one in $T$, because $k\ge3$ there exist vertex or vertices with degree greater than $2$ in $T$, now we choose one of them with minimum distance from $v$ and call it $w$. Consider $vv_1v_2\ldots v_{m-1}w$ as the unique path from $v$ to $w$ in $T$ and put $e_1=vv_1, e_2=v_1v_2,\ldots, e_m=v_{m-1}w$.
\end{proof}

It is obvious every cycle $C_n$($n\geq2$) has a spanning $2$-end tree and if in cycle $C_n$ with $n\geq3$ we contract an edge then we have $C_{n-1}$. It is interesting to know for which edges of a graph $G$ if we contract each of them, for example $e$, then for the minimum $k$ such that $G/e$ has a spanning $k$-end tree, $G$ also has a spanning $k$-end tree. In cycle $C_n$ for every edge $e=uv$ have $|N_e(u)-N_e(v)|\leq1$ and $|N_e(v)-N_e(u)|\leq1$. At the following theorem we prove that, if these two inequalities hold in a graph then we can conclude our above ideal result.

\begin{theorem}
Suppose in a connected graph $G$ for an edge $e$ with end vertices $u$ and $v$ we have $|N_e(u)-N_e(v)|\leq1$ and $|N_e(v)-N_e(u)|\leq1$. If $G/e$ has a spanning $k$-end tree then $G$ has a spanning $k$-end tree.
\end{theorem}
\begin{proof}
Suppose $T$ is a spanning $k$-end tree of $G/e$ and $w$ is the vertex of $G/e$ that produced by contracting edge $e$ by identifying vertices $v$ and $u$.\\
If $\deg_T w=1$ and $wy\in E(T)$ then $vy\in E(G)$ or $uy\in E(G)$, if $vy\in E(G)$ then we make a subdivision of $T$ with a new vertex(called $v$) on edge $wy$ and rename $w$ to $u$, then this subdivision is a spanning $k$-end tree of $G$, and if $uy\in E(G)$ do similar.\\
If $\deg_T w>1$ then at least one of it's $|N_T(w)|$ adjacent vertices is adjacent to $v$ in $G$ and at least one of them is adjacent to $u$ in $G$. Now we replace the $w$ in $T$ with the edge $e$ and consider $N_T(w)$, we can choose $x_1$,$x_2$ $\in N_T(w)$ such $x_1v\in E(G)$ and $x_2u\in E(G)$ and draw these two edges and for other vertices in $N_T(w)$ we connect each one to just one of $u$ and $v$ such that they are adjacent in $G$. Now we have a spanning $k$-end tree of $G$.

\end{proof}

\begin{corollary}
Suppose a graph $G$ has a spanning $k$-end tree such that $k$ is as minimal as possible, then there is no edge $e$ such $G/e$ has a spanning $p$-ended tree, where $p<k-1$.
\end{corollary}
\begin{proof}
If $T$ is a spanning $p$-ended tree of $G/e$ where $p<k-1$ then like proof of Theorem 2.3 if w is the vertex of $G/e$ produced by contraction on edge $e$ with end vertices $u$ and $v$, if we replace $w$ with edge $e$ and each vertex in $N_T(w)$ connect to just one of $u$ and $v$ such they are adjacent in $G$ the the new graph is a spanning $p+1$-ended tree of $G$ and this contradicts with minimality of $k$.
\end{proof}
\bigskip
\section*{Acknowledgement} I would like to thanks Daniel Yaqubi for offering various helpful suggestions which helped to make the presentation much clear.

\hspace{1in}

\bibliographystyle{amsplain}

\end{document}